\documentclass[preprint,review,10pt]{elsarticle}
\usepackage{dirtytalk}
\usepackage{ifpdf}
\usepackage{graphicx}
\usepackage{graphicx}
\usepackage{amssymb}
\usepackage[margin=.9in]{geometry}
\small\normalsize
\usepackage{amsmath}

\usepackage{amsmath,amsthm,amssymb,enumerate,amsfonts}
\newtheoremstyle{lpthm}{\baselineskip}{\baselineskip}{\slshape}{}{\bfseries}{.}{
}{}

\newtheorem{theorem}{Theorem}[section]
\newtheorem{lemma}[theorem]{Lemma}
\newtheorem{corollary}{Corollary}[section]

\newtheorem{definition}[theorem]{Definition}
\newtheorem{prop}[theorem]{Proposition}
\newtheorem{example}[theorem]{Example}

\journal{XXXX}
\begin{document}

\begin{frontmatter}

\title{Convergence of algorithms for fixed points of relatively nonexpansive mappings via
Ishikawa iteration}
\author[a]{V.~Pragadeeswarar}
\ead{v\_pragadeeswarar@cb.amrita.edu}
\cortext[auth1]{Corresponding author; Phone : +91 96593 11515.}
\author[a]{R.~Gopi}
\author[b]{Choonkil Park}
\author[c]{Dong Yun Shin}
\address[a]{Department of Mathematics, Amrita School of Engineering, Amrita Vishwa Vidyapeetham,
Coimbatore-641112, Tamil Nadu, India.}
\address[b]{ Research Institute for Natural Sciences, Hanyang University, Seoul 04763, Korea}

\address[c]{Department of Mathematics, University of Seoul, Seoul 02504, Korea.}
\date{\today.}

\begin{abstract}
By using the Ishikawa iterative algorithm, we approximate the fixed points and the best proximity points of a 
relatively nonexpansive mapping. Also, we use the von Neumann sequence to prove the convergence result in a Hilbert space setting.
A comparison table is prepared using a numerical example which shows that the Ishikawa iterative algorithm is faster 
than some known iterative algorithms such as Picard and Mann iteration.
\end{abstract}

\end{frontmatter}

\section{Introduction and Preliminaries}

Let $K$ be a nonempty subset of a Banach space $X$. The map $F : K \rightarrow K$ is non expansive if $\parallel Fw - Fz \parallel \leq 
\parallel w - z \parallel$ for all $w,z \in K$. In 1967, Browder \cite{e1} constructed the iterative process to fixed points of non expansive 
self maps on closed and convex subsets of a Hilbert space. Recently, many researchers are interested to study about the convergence of fixed points 
for these kind of mappings via different types of iterative process.  In \cite{1}, the authors have derived the results on convergence of Mann iteration 
process $w_{n+1}=(1-\eta_n)w_n + \eta_n Fw_n, \, \eta_n \in (\epsilon, 1-\epsilon)$ to relatively non 
expansive map of the type $F: M\cup N \rightarrow M\cup N,$ which satisfies $(i) \,F(M)\subseteq M$ 
and $F(N) \subseteq N$ and $(ii) \, \left\| Fw-Fz\right\| \leq \left\| w-z \right\|, \, \forall w \in M, z \in N.$ 
To prove these results, the authors used the von Neumann sequences. One can note that, a relatively non expansive
mappings need not be continuous in general.

Inspired by the work of Anthony Eldred et al. \cite{1}, in this paper, we obtain the 
convergence results of Ishikawa iteration process for relatively 
non expansive mappings in the Hilbert space setting via von Neumann sequences.

We also propose a numerical example to show that the Ishikawa iterative process converges more effectively than the Picard iterative process
and Mann iterative process.

\par The following notations are used subsequently:
\begin{eqnarray*}
	& & P_M(w)  = \{ z \in M : \left\| w-z \right\| = d(w,M) \};\\
	& & d(M,N)  =  inf\{ \left\| w-z \right\| : w \in M, z\in N \};\\
	& & M_0  =  \{ w\in M: \left\| w-z'\right\| = d(M,N) \, \text{for some} \, z'\in N \};\\
	& & N_0  = \{ z\in N: \left\| w'-z\right\| = d(M,N) \, \text{for some} \, w' \in M \}.
\end{eqnarray*}    

\par If $M$ is convex, closed subset of a reflexive and strictly convex space, then $P_M(w)$ contains one element and if $M$ and $N$ are convex,
closed subsets of a reflexive space, with either $M$ or $N$ is bounded, then $M_0 \neq \emptyset.$

\par The following definitions and theorems are very useful to our results:

\begin{definition}
	Let $M$ and $N$ be nonempty subsets of a metric space $(X,d).$ An element $w \in M$ is said to be a best proximity point of the nonself-mapping $F: M \rightarrow N$ if it satisfies the condition that 
	\begin{eqnarray*}
		d(w,Fw) = d(M,N).
	\end{eqnarray*}
\end{definition}

\begin{definition}
	Let $M$ and $N$ be nonempty subsets of a Banach space $X$. A mapping $F : M\cup N \rightarrow M\cup N$ is relatively non expansive if
	\begin{eqnarray*}
		\left\| Fw-Fz\right\| \leq \left\| w-z \right\|, \, \text{for all} \,\, w\in M, z\in N.
	\end{eqnarray*} 
\end{definition}

\begin{theorem}\cite{2} \label{l1}
	Let $M$ and $N$ be nonempty closed bounded convex subsets of a uniformly convex Banach space. Let $F : M\cup N \rightarrow M\cup N$ satisfies
	\begin{enumerate}
		\item $F(M)\subseteq N$ and $F(N)\subseteq M;$ and
		\item $\left\| Fw-Fz\right\| \leq \left\| w-z \right\|$ for $w\in M, z\in N.$
	\end{enumerate}
	Then there exist $(w,z) \in M\times N$ such that $\left\| w-Fw\right\| = \left\| z-Fz \right\| = d(M,N).$
\end{theorem}

\begin{theorem}\cite{2} \label{l2}
	Let $M$ and $N$ be nonempty closed bounded convex subsets of a uniformly convex Banach space. Let $F : M\cup N \rightarrow M\cup N$ satisfies
	\begin{enumerate}
		\item $F(M)\subseteq M$ and $F(N)\subseteq N;$ and
		\item $\left\| Fw-Fz\right\| \leq \left\| w-z \right\|$ for $w\in M, z\in N.$
	\end{enumerate}
	Then there exist $w_0 \in M$ and $z_0 \in N$ such that $Fw_0=w_0, Fz_0=z_0,$ and $\left\| w_0-z_0\right\| = d(M,N).$
\end{theorem}

\begin{theorem}\cite{9}\label{ll1}
	Let $X$ be a uniformly convex Banach space, $F$ is a non expansive mapping of the closed convex bounded subset $K$ of $X$ into $K$. 
	Then $F$ has a fixed point in $K$.
\end{theorem}

\begin{prop}\cite{5} \label{l5}
	If $X$ is a uniformly convex space and $\eta \in (0,1)$ and $\epsilon > 0,$ 
	then for any $d>0,$ if $w,z \in X$ are such that $\left\| w\right\| \leq d, \left\| z\right\| \leq d, \left\| w-z\right\| \geq \epsilon,$ 
	then there exists $\delta = \delta (\frac{\epsilon}{d}) > 0$ such that $\left\| \eta w+(1-\eta)z \right\| \leq \Big(1-2\delta(\frac{\epsilon}{d})
	min(\eta, 1-\eta) \Big)d.$
\end{prop}

\begin{lemma}\cite{10}\label{lem1}
	Suppose $X$ is a uniformly convex Banach space. Suppose $0 < a < b< 1,$ and $\{t_n\}$ is a sequence in $[a,b].$  Suppose 
	$\{ w_n \} $ and $\{ z_n\}$ are sequences in $X$ such that $\parallel w_n \parallel \leq 1, \parallel z_n \parallel \leq 1$ for all $n$. 
	Define $\{ a_n\}$ in $X$ by $ a_n = (1-t_n)w_n + t_n z_n.$ If $\lim_{n\to\infty} \parallel a_n \parallel = 1,$ then $\lim_{n\to\infty} \parallel w_n - z_n \parallel = 0.$
\end{lemma}

We prove the following result which shows that, If $F$ is a nonexpansive mapping then the Ishikawa iteration converges to a fixed point of $F$. 
Moreover, it is useful to prove our main results.

\begin{theorem} \label{l3}
	Let $K$ be a nonempty bounded closed convex subset of a uniformly convex Banach space $X$ 
	and suppose $F : K\rightarrow K$ is a non expansive mapping.
	Let $w_0 \in K$ and define $w_{n+1}=(1-\eta_n)w_n+\eta_n F((1-\delta_n)w_n+\delta_n Fw_n), \,\text{where} \,\, \eta_n, \delta_n \in 
	(\epsilon, 1-\epsilon), n =0,1,2,...$ and $\epsilon\in(0,\frac{1}{2})$. Then $\lim_{n \rightarrow \infty} \left\| w_n-Fw_n\right\| = 0.$ 
	Moreover, if $F(K)$ lies in a compact set, $\{w_n \}$ converges to a fixed point of $F$. 
\end{theorem}
\begin{proof}
	By Theorem \ref{ll1}, there exist $z \in X$ such that  $Fz = z.$ Now,
	\begin{eqnarray*}
		\parallel w_{n+1}- z\parallel & = & \parallel(1-\eta_n)w_n + \eta_n F\big((1-\delta_n)w_n+\delta_n Fw_n\big)-z\parallel \\
		& = & \parallel (1-\eta_n)w_n + \eta_n F\big((1-\delta_n)w_n+\delta_n Fw_n\big)-\big((1-\eta_n)z + \eta_n z \big) \parallel \\
		& \leq & (1-\eta_n) \parallel w_n-z \parallel + \eta_n \parallel F\big((1-\delta_n)w_n + \delta_n Fw_n\big)-Fz \parallel \\
		& \leq & (1-\eta_n) \parallel w_n-z \parallel + \eta_n \parallel (1-\delta_n)w_n + \delta_n Fw_n-z \parallel \\
		& = & (1-\eta_n) \parallel w_n-z \parallel \\&& + \eta_n \parallel (1-\delta_n)w_n + \delta_n Fw_n-\big( (1-\delta_n)z + \delta_n z \big) \parallel \\
		& \leq & (1-\eta_n) \parallel w_n-z \parallel + \eta_n \big( \parallel (1-\delta_n)(w_n-z) \parallel + \delta_n \parallel Fw_n- Fz \parallel \big) \\
		& \leq & \parallel w_n-z \parallel.
	\end{eqnarray*}
	This implies that the sequence $\{ \parallel w_n-z \parallel \} $ is non increasing and bounded below by $0.$ Hence, we have 
	$\parallel w_n - z \parallel \rightarrow d\geq0.$ \\	 
	\textbf{Case (i) :} If $\parallel w_n - z \parallel \rightarrow 0.$
	\begin{eqnarray*}
		\parallel w_n - Fw_n \parallel &\leq& \parallel w_n - z \parallel + \parallel z - F w_n \parallel\\
		&=& \parallel w_n - z \parallel + \parallel Fz - F w_n \parallel\\
		&\leq& \parallel w_n - z \parallel + \parallel z - w_n \parallel.
	\end{eqnarray*}
	As $n \rightarrow \infty$, we get $\parallel w_n - Fw_n \parallel \rightarrow 0.$ Let $z_n=(1-\delta_n) w_n + \delta_n Fw_n.$
	\begin{eqnarray*}
		\parallel z_n - z \parallel &=& \parallel (1-\delta_n) w_n + \delta_n Fw_n - z \parallel \\
		&= & \parallel (1-\delta_n) w_n + \delta_n Fw_n - ((1-\delta_n)z + \delta_n z) \parallel \\
		&\leq& (1-\delta_n) \parallel w_n - z \parallel + \delta_n \parallel Fw_n - Fz \parallel \\
		&\leq & (1-\delta_n) \parallel w_n - z \parallel + \delta_n \parallel w_n - z \parallel \\
		&=& \parallel w_n - z \parallel.
	\end{eqnarray*}
	And we know that 
	\begin{eqnarray*}
		\parallel w_{n+1} - w_n \parallel &=& \eta_n \parallel Fz_n - w_n \parallel\\
		&\leq& \eta_n (\parallel Fz_n - z \parallel + \parallel z - w_n \parallel)\\
		&\leq & \eta_n (\parallel z_n - z \parallel + \parallel z - w_n \parallel)
	\end{eqnarray*}
	\begin{eqnarray*}
		&\leq& \eta_n (\parallel w_n - z \parallel + \parallel z - w_n \parallel).
	\end{eqnarray*}
	As $n \rightarrow \infty,$ we get $\parallel w_{n+1} - w_n \parallel \rightarrow 0.$\\
	\textbf{Case (ii) :} If $\parallel w_n - z \parallel \rightarrow d > 0.$ We need to show that $\parallel w_{n} - Fw_{n} \parallel\to0.$ Suppose not.
	Then there exists a subsequence $\{ w_{n_k} \}$ of $\{ w_n \}$ and an $\epsilon>0$ such that $\parallel w_{n_k} - Fw_{n_k} \parallel \geq \epsilon > 0$ for all $k$.\\
	Since the modulus of convexity of $\delta$ of $X$ is continuous and  increasing function we choose $\xi > 0$ 
	as small that $\Big( 1- c\delta \big(\frac{\epsilon}{d+\xi} \big) \Big)(d+\xi)< d,$ where $c > 0.$\\
	Now we choose $k$, such that $\parallel w_{n_k} - z \parallel \leq d+ \xi.$ By using Proposition \ref{l5},
	\begin{eqnarray*}
		\parallel z-w_{n_k + 1} \parallel & = & \parallel z - \big( (1-\eta_{n_k})w_{n_k} + \eta_{n_k} F\big((1-\delta_{n_k})w_{n_k} +\delta_{n_k} Fw_{n_k}\big) \big) \parallel \\
		& = & \parallel (1-\eta_{n_k})z + \eta_{n_k}z \\&& - \big( (1-\eta_{n_k})w_{n_k} + \eta_{n_k} F\big((1-\delta_{n_k})w_{n_k} +\delta_{n_k} Fw_{n_k}\big) \big) \parallel \\
		& \leq & (1-\eta_{n_k}) \parallel z - w_{n_k} \parallel + \eta_{n_k} \parallel Fz - F\big((1-\delta_{n_k})w_{n_k} +\delta_{n_k} Fw_{n_k}\big) \parallel \\
		& \leq & (1-\eta_{n_k}) (d+ \xi) + \eta_{n_k} \parallel z - \big((1-\delta_{n_k})w_{n_k} +\delta_{n_k} Fw_{n_k}\big) \parallel \\
		& = & (1-\eta_{n_k}) (d+ \xi) + \eta_{n_k} \parallel (1-\delta_{n_k})(z - w_{n_k}) +\delta_{n_k} (z-Fw_{n_k}) \parallel \\
		& \leq & (1-\eta_{n_k}) (d+ \xi) + \eta_{n_k} \Big( 1-2 \delta \Big(\frac{\epsilon}{d+\xi} \Big) \text{min} \{ \delta_{n_k}, 1-\delta_{n_k} \} \Big) (d+\xi)\\
		& = & \Big(1-\eta_{n_k} + \eta_{n_k} - 2 \eta_{n_k} \delta \Big(\frac{\epsilon}{d+\xi} \Big) \text{min} \{ \delta_{n_k}, 1-\delta_{n_k} \} \Big) (d+\xi)\\
		& = & \Big(1 - 2 \delta \Big(\frac{\epsilon}{d+\xi} \Big) \text{min} \{ \eta_{n_k} \delta_{n_k}, \eta_{n_k}(1-\delta_{n_k}) \} \Big) (d+\xi).
	\end{eqnarray*}
	Since there exists $l > 0$ such that $2 \,\text{min} \{ \eta_{n_k} \delta_{n_k}, \eta_{n_k}(1-\delta_{n_k}) \} \geq l,$
	\begin{eqnarray*}
		\Big(1 - 2 \delta \Big(\frac{\epsilon}{d+\xi} \Big) \text{min} \{ \eta_{n_k} \delta_{n_k}, \eta_{n_k}(1-\delta_{n_k}) \} \Big) (d+\xi) \leq \Big( 1- l\delta \Big(\frac{\epsilon}{d+\xi} \Big)\Big)(d +\xi).
	\end{eqnarray*}
	Suppose we choose very small $\xi > 0$, we have $\Big( 1- l\delta \Big(\frac{\epsilon}{d+\xi} \Big)\Big)(d +\xi) < d$, which is contradiction.
	This implies that $\lim_{n\to\infty}\parallel w_n-Fw_n \parallel = 0.$\\ Now we prove that $\| w_{n+1} - w_n \parallel \rightarrow 0.$ 
	We have $\parallel w_{n+1} - w_n \parallel = \eta_n \parallel Fz_n - w_n \parallel$,
	where $z_n = (1-\delta_n)w_n + \delta_n Fw_n.$ 
	Now, we define $a_n = \frac{w_{n+1} - z}{\parallel w_n - z \parallel}$, \, $v_n = \frac{Fz_n - z}{\parallel w_n - z \parallel}$ and
	$w_n = \frac{w_n - z}{\parallel w_n - z
		\parallel}.$ One can note that $\parallel w_n \parallel = 1.$ Now,
	\begin{eqnarray*}
		\parallel Fz_n - z \parallel &=& \parallel Fz_n - Fz \parallel \\
		&\leq& \parallel z_n - z \parallel \\
		&\leq& \parallel (1-\delta_n) w_n + \delta_n Fw_n - z \parallel \\
		&\leq & \parallel (1-\delta_n) w_n + \delta_n Fw_n - ((1-\delta_n)z + \delta_n z) \parallel \\
		&\leq& (1-\delta_n) \parallel w_n - z \parallel + \delta_n \parallel Fw_n - Fz \parallel \\
		&\leq & (1-\delta_n) \parallel w_n - z \parallel + \delta_n \parallel w_n - z \parallel \\
		&=& \parallel w_n - z \parallel. 
	\end{eqnarray*}
	Therefore $\parallel v_n \parallel = \frac{\parallel Fz_n - z \parallel}{\parallel w_n - z \parallel} \leq \frac{\parallel w_n - z \parallel}{\parallel w_n - z \parallel} = 1.$ 
	From the Ishikawa iteration, we obtain $w_{n+1} - z = (1-\eta_n)(w_n - z) + \eta_n (Fz_n - z)$. 
	Dividing by $\parallel w_n - z \parallel,$ we get $$\frac{w_{n+1}- z}{\parallel w_n - z \parallel} = (1-\eta_n) \frac{(w_n - z)}{\parallel w_n - z \parallel} + \eta_n \frac{(Fz_n - z)}{\parallel w_n - z \parallel}.$$ 
	Then $a_n = (1-\eta_n)w_n + \eta_n v_n.$ 
	Now we prove that $\parallel a_n \parallel \rightarrow 1.$ Now,
	\begin{eqnarray*}
		\lim_{n \rightarrow \infty} \parallel a_n \parallel
		&=& \lim_{n\rightarrow \infty} \frac{\parallel w_{n+1} - z \parallel}{\parallel w_n - z \parallel} = \frac{d}{d} = 1. 
	\end{eqnarray*}
	By Lemma \ref{lem1}, $\parallel w_n - v_n \parallel \rightarrow 0.$ This implies that $\parallel w_n - Fz_n \parallel \rightarrow 0.$ Therefore $\parallel w_{n+1} - w_n \parallel \rightarrow 0.$\\
	Since $F(K)$ is contained in a compact set, $\{Fw_n \}$ has a subsequence $\{ Fw_{n_k} \}$ that converges to a point $a \in M.$ 
	Also $\{ w_{n_k} \}$ and $\{ w_{n_k+1} \}$ converge to $a$. 
	This implies that $\{w_n\} $ converge to $a.$ Then $Fw_n \rightarrow a.$ 
	In particular, $ Fw_{n_k}  \rightarrow a$ and $w_{n_k} \rightarrow a.$ Since $F$ is continuous, 
	implies that $ Fw_{n_k}  \rightarrow Fa.$ Therefore $Fa = a.$
\end{proof}

\begin{theorem}\cite{1} \label{l4}
	Let $M$ and $N$ be nonempty bounded closed convex subset of a uniformly convex Banach space and suppose $F: M\cup N \rightarrow M\cup N$ satisfies
	\begin{enumerate}
		\item $F(M) \subseteq M$ and $F(N) \subseteq N ;$ and
		\item $\parallel Fw-Fz \parallel \leq \parallel w-z \parallel $ for $w\in M, z\in N.$
	\end{enumerate}
	Let $w_0\in M,$ and 
	define $w_{n+1}=P^n \big( (1-\eta_n)w_n + \eta_n Fw_n \big),\, \eta_n \in (\epsilon, 1-\epsilon),$ where $\epsilon\in (0,1/2)$ and $n=0,1,2,...$. Then $lim_{n \rightarrow \infty}\parallel w_n-Fw_n \parallel = 0.$ Moreover, if $F(M)$ lies in a compact set, then $\{w_n\}$ converges to a fixed point of $F$.
\end{theorem}

\par Let $M$ be a convex closed subset of a Hilbert Space $X$. Then for $w\in X,$ we know that $ P_M(w)$ is the nearest to $w$ and unique point of $M$.
And also $P_M$ is non expansive and distinguished by the Kolmogorov's criterion:\\
$\left\langle w-P_Mw, P_Mw-a\right\rangle \geq 0,$ for all $w\in X$ and $a\in M.$

\par Let $M$ and $N$ be two convex closed subsets of $X.$ Define
\begin{eqnarray*}
	P(w) = P_M(P_N(w)) \, \text{for each} \, w \in X,
\end{eqnarray*} 
then the sequences $\{P^n(w)\} \subset M$ and $\{ P_N(P^n(w)) \} \subset  N.$ When $M$ and $N$ are closed, 
the convergence of these sequences in norm were proved by von Neumann \cite{8}. 
The sequences $\{P^n(w)\}$ and $\{ P_N(P^n(w)) \}$ 
are called von  Neumann  sequences or  alternating  projection  algorithm for two sets.

\begin{definition}\cite{4}
	Let $M$ and $N$ be nonempty closed convex subsets of a Hilbert space $X$. We say that $(M,N)$ is boundedly regular if for each bounded subset $S$ of $X$ and for each $\epsilon>0$ there exist $\delta>0$ such that 
	\begin{eqnarray}
	max\{d(w,M),d(w,N-v)\}\leq \delta \Rightarrow d(w,N)\leq \epsilon,\,\, \forall w\in X,
	\end{eqnarray}
	where $v=P_{\overline{N-M}}(0),$ the displacement vector from $M$ to $N$. $(v$ is the unique vector satisfying $\parallel v\parallel= d(M,N))$.
\end{definition} 

\begin{theorem}\cite{4}\label{l6}
	If $(M,N)$ is boundedly regular, then the von Neumann sequences converges in norm.
\end{theorem}

\begin{theorem}\cite{4} \label{l7}
	If $M$ or $N$ is boundedly compact, then $(M,N)$ is boundedly regular.
\end{theorem}

\begin{lemma}\cite{3} \label{l8}
	Let $M$ be a nonempty closed and convex subset and $N$ be nonempty closed subset of a uniformly convex Banach space. Let $\{w_n\}$ and $\{a_n\}$ be sequences in $M$ and $\{z_n\}$ be a sequence in $N$ satisfying:
	\begin{enumerate}
		\item $\parallel w_n-z_n\parallel\rightarrow d(M,N), \text{and}$
		\item $\parallel a_n-z_n\parallel\rightarrow d(M,N). \text{Then} \parallel w_n-a_n \parallel \text{converges to zero}.$
	\end{enumerate}
\end{lemma}

\begin{corollary}\cite{3} \label{l9}
	Let $M$ be a nonempty closed convex subset and $N$ be a nonempty closed subset of  uniformly convex Banach space. Let $\{ w_n \}$ be a sequence in $M$ and $z_0 \in N$ 
	such that $\parallel w_n-z_0 \parallel \rightarrow d(M,N).$ Then $\{ w_n \}$ converges to $P_M(z_0).$
\end{corollary}

\begin{prop}\cite{2} \label{l10}
	Let $M$ and $N$ be two closed and convex subsets of a Hilbert space $X$. Then $P_N(M)\subseteq N, P_M(N)\subseteq M,$ and $\parallel P_Nw-P_Mz\parallel\leq \parallel w-z \parallel $ for $w\in M$ and $z\in N$.
\end{prop}

\begin{lemma} \label{l11}
	Let $M$ and $N$ be two closed and convex subsets of a Hilbert space $X$. For each $w\in X,$
	\begin{eqnarray*}
		\parallel P^{n+1}(w)-a\parallel \leq \parallel P^{n}(w)-a \parallel, \text{for each} \,\,  a \in M_0\cup N_0.
	\end{eqnarray*}
\end{lemma}

\section{Main Results}

\begin{theorem} \label{l12}
	Let $M$ and $N$ be nonempty bounded closed convex subsets of a uniformly convex Banach space and suppose $F: M\cup N \rightarrow M\cup N$ satisfies
	\begin{enumerate}
		\item $F(M) \subseteq M$ and $F(N) \subseteq N ;$ and
		\item $\parallel Fw-Fz \parallel \leq \parallel w-z \parallel $ for $w\in M, z\in N.$
	\end{enumerate}
	Let $w_0\in M,$ and define $w_{n+1}=(1-\eta_n)w_n + \eta_n F \big((1-\delta_n)w_n+\delta_n Fw_n \big),\, \eta_n, \delta_n \in (\epsilon, 1-\epsilon),$ where
	$\epsilon\in (0,1/2)$ and $n=0,1,2,...$. Suppose $d(w_n, M_0) \rightarrow 0,$ then $lim_{n \rightarrow \infty}\parallel w_n-Fw_n \parallel = 0.$ Moreover, if $F(M)$ lies in a compact set, then $\{w_n\}$ converges to a fixed point of $F$.
\end{theorem}

\begin{proof}
	If $d(M,N) = 0,$ then $M_0 = N_0 = M\cap N$ and by Theorem \ref{l3} we can prove the result from the truth that $F: M\cap N \rightarrow M\cap N$ is nonexpansive. 
	Therefore let us take that $d(M,N)>0.$ By Theorem \ref{l2}, there exists $z\in N_0$ such that $Fz = z.$ Now,
	\begin{eqnarray*}
		\parallel w_{n+1}- z\parallel & = & \parallel(1-\eta_n)w_n + \eta_n F\big((1-\delta_n)w_n+\delta_n Fw_n\big)-z\parallel \\
		& = & \parallel (1-\eta_n)w_n + \eta_n F\big((1-\delta_n)w_n+\delta_n Fw_n\big)-\big((1-\eta_n)z + \eta_n z \big) \parallel \\
		& \leq & (1-\eta_n) \parallel w_n-z \parallel + \eta_n \parallel F\big((1-\delta_n)w_n + \delta_n Fw_n\big)-Fz \parallel \\
		& \leq & (1-\eta_n) \parallel w_n-z \parallel + \eta_n \parallel (1-\delta_n)w_n + \delta_n Fw_n-z \parallel \\
		& = & (1-\eta_n) \parallel w_n-z \parallel \\&& + \eta_n \parallel (1-\delta_n)w_n + \delta_n Fw_n-\big( (1-\delta_n)z + \delta_n z \big) \parallel \\
		& \leq & (1-\eta_n) \parallel w_n-z \parallel + \eta_n \big( \parallel (1-\delta_n)(w_n-z) \parallel + \delta_n \parallel Fw_n- Fz \parallel \big) \\
		& \leq & \parallel w_n-z \parallel.
	\end{eqnarray*}
	This implies that the sequence $\{ \parallel w_n-z \parallel \} $ is non increasing. Then we can find $d>0$ such that $\lim_{n \rightarrow \infty} \parallel w_n-z \parallel = d.$ \\
	Suppose there exists a subsequence $\{ w_{n_k} \}$ of $\{ w_n \}$ and an $\epsilon>0$ such that \\ $\parallel w_{n_k} - Fw_{n_k} \parallel \geq \epsilon > 0$ for all $k$.\\
	Since the modulus of convexity of $\delta$ of $X$ is continuous and  increasing function we choose $\xi > 0$ as small that $\Big( 1- c\delta \big(\frac{\epsilon}{d+\xi} \big) \Big)(d+\xi)< d,$ where $c > 0.$\\
	Now we choose $k$, such that $\parallel w_{n_k} - z \parallel \leq d+ \xi.$ By using Proposition \ref{l5},
	\begin{eqnarray*}
		\parallel z-w_{n_k + 1} \parallel & = & \parallel z - \big( (1-\eta_{n_k})w_{n_k} + \eta_{n_k} F\big((1-\delta_{n_k})w_{n_k} +\delta_{n_k} Fw_{n_k}\big) \big) \parallel \\
		& = & \parallel (1-\eta_{n_k})z + \eta_{n_k}z \\&& - \big( (1-\eta_{n_k})w_{n_k} + \eta_{n_k} F\big((1-\delta_{n_k})w_{n_k} +\delta_{n_k} Fw_{n_k}\big) \big) \parallel \\
		& \leq & (1-\eta_{n_k}) \parallel z - w_{n_k} \parallel + \eta_{n_k} \parallel Fz - F\big((1-\delta_{n_k})w_{n_k} +\delta_{n_k} Fw_{n_k}\big) \parallel \\
		& \leq & (1-\eta_{n_k}) (d+ \xi) + \eta_{n_k} \parallel z - \big((1-\delta_{n_k})w_{n_k} +\delta_{n_k} Fw_{n_k}\big) \parallel \\
		& = & (1-\eta_{n_k}) (d+ \xi) + \eta_{n_k} \parallel (1-\delta_{n_k})(z - w_{n_k}) +\delta_{n_k} (z-Fw_{n_k}) \parallel \\
		& \leq & (1-\eta_{n_k}) (d+ \xi) + \eta_{n_k} \Big( 1-2 \delta \Big(\frac{\epsilon}{d+\xi} \Big) \text{min} \{ \delta_{n_k}, 1-\delta_{n_k} \} \Big) (d+\xi)\\
		& = & \Big(1-\eta_{n_k} + \eta_{n_k} - 2 \eta_{n_k} \delta \Big(\frac{\epsilon}{d+\xi} \Big) \text{min} \{ \delta_{n_k}, 1-\delta_{n_k} \} \Big) (d+\xi)\\
		& = & \Big(1 - 2 \delta \Big(\frac{\epsilon}{d+\xi} \Big) \text{min} \{ \eta_{n_k} \delta_{n_k}, \eta_{n_k}(1-\delta_{n_k}) \} \Big) (d+\xi).
	\end{eqnarray*}
	Since there exists $l > 0$ such that $2 \,\text{min} \{ \eta_{n_k} \delta_{n_k}, \eta_{n_k}(1-\delta_{n_k}) \} \geq l,$
	\begin{eqnarray*}
		\Big(1 - 2 \delta \Big(\frac{\epsilon}{d+\xi} \Big) \text{min} \{ \eta_{n_k} \delta_{n_k}, \eta_{n_k}(1-\delta_{n_k}) \} \Big) (d+\xi) \leq \Big( 1- l\delta \Big(\frac{\epsilon}{d+\xi} \Big)\Big)(d +\xi).
	\end{eqnarray*}
	Suppose we choose very small $\xi > 0$, we have $\Big( 1- l\delta \Big(\frac{\epsilon}{d+\xi} \Big)\Big)(d +\xi) < d$, which is contradiction. 
	This implies that $lim_{n \rightarrow \infty}\parallel w_n-Fw_n \parallel = 0.$ 
	Now we prove that $\| w_{n+1} - w_n \parallel \rightarrow 0.$ 
	We have $\parallel w_{n+1} - w_n \parallel = \eta_n \parallel Fz_n - w_n \parallel$,
	where $z_n = (1-\delta_n)w_n + \delta_n Fw_n.$ 
	Now, we define $a_n = \frac{w_{n+1} - z}{\parallel w_n - z \parallel}$, \, $v_n = \frac{Fz_n - z}{\parallel w_n - z \parallel}$ and
	$b_n = \frac{w_n - z}{\parallel w_n - z
		\parallel}.$ One can note that $\parallel b_n \parallel = 1.$ Now,
	\begin{eqnarray*}
		\parallel Fz_n - z \parallel &=& \parallel Fz_n - Fz \parallel \\
		&\leq& \parallel z_n - z \parallel \\
		&\leq& \parallel (1-\delta_n) w_n + \delta_n Fw_n - z \parallel \\
		&\leq & \parallel (1-\delta_n) w_n + \delta_n Fw_n - ((1-\delta_n)z + \delta_n z) \parallel \\
		&\leq& (1-\delta_n) \parallel w_n - z \parallel + \delta_n \parallel Fw_n - Fz \parallel \\
		&\leq & (1-\delta_n) \parallel w_n - z \parallel + \delta_n \parallel w_n - z \parallel \\
		&=& \parallel w_n - z \parallel. 
	\end{eqnarray*}
	Therefore $\parallel v_n \parallel = \frac{\parallel Fz_n - z \parallel}{\parallel w_n - z \parallel} \leq \frac{\parallel w_n - z \parallel}{\parallel w_n - z \parallel} = 1.$ 
	From the Ishikawa iteration, we obtain $w_{n+1} - z = (1-\eta_n)(w_n - z) + \eta_n (Fz_n - z)$. 
	Dividing by $\parallel w_n - z \parallel,$ we get $$\frac{w_{n+1}- z}{\parallel w_n - z \parallel} = (1-\eta_n) \frac{(w_n - z)}{\parallel w_n - z \parallel} + \eta_n \frac{(Fz_n - z)}{\parallel w_n - z \parallel}.$$ Then $a_n = (1-\eta_n)b_n + \eta_n v_n.$ 
	Now we prove that $\parallel a_n \parallel \rightarrow 1.$ Now
	\begin{eqnarray*}
		\lim_{n \rightarrow \infty} \parallel a_n \parallel
		&=& \lim_{n\rightarrow \infty} \frac{\parallel w_{n+1} - z \parallel}{\parallel w_n - z \parallel} = \frac{d}{d} = 1. 
	\end{eqnarray*}
	By Lemma \ref{lem1}, $\parallel b_n - v_n \parallel \rightarrow 0.$ This implies that $\parallel w_n - Fz_n \parallel \rightarrow 0.$
	Since $F(M)$ is contained in a compact set, $\{Fw_n \}$ has a subsequence $\{ Fw_{n_k} \}$ that converges to a point $a \in M.$ Also $\{ w_{n_k} \}$ and $\{ w_{n_k+1} \}$ converge to $a$.\\
	Since $d(w_n, M_0) \rightarrow 0,$ there exist $\{ a_n \} \subseteq M_0,$ such that $\parallel w_n - a_n \parallel \rightarrow 0.$ Therefore, $a_{n_k} \rightarrow a,$
	which gives that $a \in M_0.$\\
	Let $D = d(M,N)$ and choose $b \in N_0$ such that $\parallel a-b \parallel = D.$ \\
	We have $ \parallel w_{n_k} - b \parallel \rightarrow \parallel a-b \parallel = D,$ and $\parallel w_{n_k} - b \parallel \geq \parallel Fw_{n_k} - Fb \parallel \rightarrow \parallel a-Fb \parallel,$ So $\parallel a-Fb \parallel = D.$ By strict convexity of the norm, $Fb=b.$ It follows that $Fa=a$.
\end{proof}

\begin{corollary} 
	Let $M$ and $N$ be nonempty bounded closed convex subsets of a uniformly convex Banach space and suppose $F: M\cup N \rightarrow M\cup N$ satisfies
	\begin{enumerate}
		\item $F(M) \subseteq M$ and $F(N) \subseteq N ;$ and
		\item $\parallel Fw-Fz \parallel \leq \parallel w-z \parallel $ for $w\in M, z\in N.$
	\end{enumerate}
	Let $w_0\in M_0,$ and define $w_{n+1}=(1-\eta_n)w_n + \eta_n F \big((1-\delta_n)w_n+\delta_n Fw_n \big),\, \eta_n, \delta_n \in (\epsilon, 1-\epsilon),$ where
	$\epsilon\in (0,1/2)$ and $n=0,1,2,...,$ then $lim_{n \rightarrow \infty}\parallel w_n-Fw_n \parallel = 0.$ Moreover, if $F(M)$ lies in a compact set, then $\{w_n\}$ converges to a fixed point of $F$.
\end{corollary}

\begin{corollary}
	Let $M$ and $N$ be nonempty bounded closed convex subsets of a Hilbert Space and Let $F$ be as in Theorem \ref{l2}.
	Let $w_0\in M_0,$ and define $w_{n+1}=P^n \big( (1-\eta_n)w_n + \eta_n Fz_n \big) ,\, $ where $  z_n = (1-\delta_n)w_n+\delta_n Fw_n,\, \eta_n, \delta_n \in (\epsilon, 1-\epsilon),$ 
	where $\epsilon\in (0,1/2)$ and $n=0,1,2,...$ then $lim_{n \rightarrow \infty}\parallel w_n-Fw_n \parallel = 0.$ Moreover, if $F(M)$ is mapped into a  compact subset of $N$, then $\{w_n\}$ converges to a fixed point of $F$.
\end{corollary}

\begin{proof}
	One can note that $P^n \big( (1-\eta_n)w_n + \eta_n Fz_n \big) = (1-\eta_n)w_n + \eta_n Fz_n,$ by Theorem \ref{l12} the result follows. 
\end{proof}

\begin{example}
	Let $X = \mathbb{R}^2,$
	\begin{eqnarray*}
		M = \{ (w,0) : -4 \leq w\leq -3 \} \, \text{and} \, N = \{ (w,0) : 3 \leq w\leq 4 \}.
	\end{eqnarray*}
	Define
	\begin{eqnarray*}
		F : M \rightarrow M \, \text{by} \, F(w,0) = \Big( \frac{w-3}{2} , 0\Big),\\
		F : N \rightarrow N \, \text{by} \, F(w,0) = \Big( \frac{w+3}{2} , 0\Big).
	\end{eqnarray*}
	Let $(w,0) \in M, (w',0) \in N.$ Then, 
	\begin{eqnarray*}
		\parallel F(w,0) - F(w',0) \parallel & = & \parallel \Big( \frac{w-3}{2} , 0\Big) - \Big( \frac{w'+3}{2} , 0\Big) \parallel\\
		& = & \parallel \Big( \frac{w-w'-6}{2} , 0 \Big) \parallel
	\end{eqnarray*}
	\begin{eqnarray*}
		& = & \sqrt{\Big( \frac{w-w'-6}{2} \Big)^2 + 0}\\
		& \leq & \sqrt{( w-w')^2}.
	\end{eqnarray*}
	Hence $F$ is a relatively non expansive mapping.\\
	Let $w_0 = -3.5$ and set $w_{n+1} = (1-\eta_n) w_n + \eta_n F\big( (1-\delta_n) w_n + \delta_n Fw_n \big)$ with $\eta_n = \delta_n = 0.999.$ We have, $Fw= \frac{w-3}{2}$.
	Then $w_{n+1} = 0.25099975 w_n - 2.24700075.$\\ 
	In Picard iteration we have $w_{n+1} = Fw_n = \frac{w_n-3}{2}$, and Mann with $\eta_n = 0.999$ or Krasnoselskij iteration , we have $w_{n+1} = (1-\eta_n) w_n + \eta_n Fw_n = 0.5005 w_n - 1.4985$. Using Matlab coding we give the comparison table for approaching fixed point in these three iteration process. 
\end{example}
Comparison of Ishikawa iteration with Mann and Picard iteration is given in the table.

The figure shows comparison of Ishikawa iteration with Mann and Picard iteration by using the continuous data points from -3.5 to -3.

In the next result, we provide a stronger version to iterate the fixed point via von Neumann sequences.

\begin{theorem}\label{l13}
	Let $M$ and $N$ be nonempty bounded closed convex subsets of a Hilbert Space and suppose $F: M\cup N \rightarrow M\cup N$ satisfies
	\begin{enumerate}
		\item $F(M) \subseteq M$ and $F(N) \subseteq N ;$ and
		\item $\parallel Fw-Fz \parallel \leq \parallel w-z \parallel $ for $w\in M, z\in N.$
	\end{enumerate}
	Let $w_0\in M,$ and define $w_{n+1}=P^n \big( (1-\eta_n)w_n + \eta_n Fz_n \big) ,\, $ where $  z_n = (1-\delta_n)w_n+\delta_n Fw_n,\, \eta_n, \delta_n \in (\epsilon, 1-\epsilon),$ 
	where $\epsilon\in (0,1/2)$ and $n=0,1,2,...$, then $lim_{n \rightarrow \infty}\parallel w_n-Fw_n \parallel = 0.$ Moreover, if $F(M)$ lies in a compact set and $\parallel w_n - F z_n \parallel \rightarrow 0$, then $\{w_n\}$ converges to a fixed point of $F$.
\end{theorem}

\begin{proof}
	If $d(M,N) = 0,$ then $M_0 = N_0 = M\cap N$ and $F: M\cap N \rightarrow M\cap N$ is non expansive with $w_{n+1}=P^n \big( (1-\eta_n)w_n + \eta_n F \big((1-\delta_n)w_n+\delta_n Fw_n \big) \big) = (1-\eta_n)w_n + \eta_n F \big((1-\delta_n)w_n+\delta_n Fw_n \big),$ the usual Ishikawa iteration. So let us take that $d(M,N)>0.$ 
	By Theorem \ref{l2},  we can find $z\in N_0$ such that $Fz = z.$ Now,
	\begin{eqnarray*}
		\parallel w_{n+1}- z\parallel & = & \parallel P^n \big( (1-\eta_n)w_n + \eta_n F\big((1-\delta_n)w_n+\delta_n Fw_n\big) \big)-z\parallel \\
		& \leq & \parallel (1-\eta_n)w_n + \eta_n F\big((1-\delta_n)w_n+\delta_n Fw_n\big) -z\parallel \\
		& = & \parallel (1-\eta_n)w_n + \eta_n F\big((1-\delta_n)w_n+\delta_n Fw_n\big)-\big((1-\eta_n)z + \eta_n z \big) \parallel \\
		& \leq & (1-\eta_n) \parallel w_n-z \parallel + \eta_n \parallel F\big((1-\delta_n)w_n + \delta_n Fw_n\big)-Fz \parallel \\
		& \leq & (1-\eta_n) \parallel w_n-z \parallel + \eta_n \parallel (1-\delta_n)w_n + \delta_n Fw_n-z \parallel \\
		& = & (1-\eta_n) \parallel w_n-z \parallel \\&& + \eta_n \parallel (1-\delta_n)w_n + \delta_n Fw_n-\big( (1-\delta_n)z + \delta_n z \big) \parallel \\
		& \leq & (1-\eta_n) \parallel w_n-z \parallel + \eta_n \big( \parallel (1-\delta_n)(w_n-z) \parallel + \delta_n \parallel Fw_n- Fz \parallel \big) \\
		& \leq & \parallel w_n-z \parallel.
	\end{eqnarray*}
	This implies that the sequence $\{ \parallel w_n-z \parallel \} $ is non increasing. Then there exists $d>0$ such that $\lim_{n \rightarrow \infty} \parallel w_n-z \parallel = d.$ \\
	Suppose there exists a subsequence $\{ w_{n_k} \}$ of $\{ w_n \}$ and an $\epsilon>0$ such that $\parallel w_{n_k} - Fw_{n_k} \parallel \geq \epsilon > 0$ for all $k$.\\
	Since the modulus of convexity of $\delta$ of $X$ is continuous and increasing function, we choose $\xi > 0$ so small that $\Big( 1- c\delta \big(\frac{\epsilon}{d+\xi} \big) \Big)(d+\xi)< d,$ where $c > 0.$\\
	Choose $k$, such that $\parallel w_{n_k} - z \parallel \leq d+ \xi.$ By using Proposition \ref{l5},
	\begin{eqnarray*}
		\parallel w_{n_k + 1} - z \parallel & = & \parallel P^{n_k} \big( (1-\eta_{n_k})w_{n_k} + \eta_{n_k} F\big((1-\delta_{n_k})w_{n_k} +\delta_{n_k} Fw_{n_k}\big) \big) - z \parallel \\
		& \leq & \parallel \big( (1-\eta_{n_k})w_{n_k} + \eta_{n_k} F\big((1-\delta_{n_k})w_{n_k} +\delta_{n_k} Fw_{n_k}\big) - z \parallel \\
		& = & \parallel (1-\eta_{n_k})z + \eta_{n_k}z \\&& - \big( (1-\eta_{n_k})w_{n_k} + \eta_{n_k} F\big((1-\delta_{n_k})w_{n_k} +\delta_{n_k} Fw_{n_k}\big) \big) \parallel \\
		& \leq & (1-\eta_{n_k}) \parallel z - w_{n_k} \parallel + \eta_{n_k} \parallel Fz - F\big((1-\delta_{n_k})w_{n_k} +\delta_{n_k} Fw_{n_k}\big) \parallel \\
		& \leq & (1-\eta_{n_k}) (d+ \xi) + \eta_{n_k} \parallel z - \big((1-\delta_{n_k})w_{n_k} +\delta_{n_k} Fw_{n_k}\big) \parallel \\
		& = & (1-\eta_{n_k}) (d+ \xi) + \eta_{n_k} \parallel (1-\delta_{n_k})(z - w_{n_k}) +\delta_{n_k} (z-Fw_{n_k}) \parallel \\
		& \leq & (1-\eta_{n_k}) (d+ \xi) + \eta_{n_k} \Big( 1-2 \delta \Big(\frac{\epsilon}{d+\xi} \Big) \text{min} \{ \delta_{n_k}, 1-\delta_{n_k} \} \Big) (d+\xi)\\
		& = & \Big(1-\eta_{n_k} + \eta_{n_k} - 2 \eta_{n_k} \delta \Big(\frac{\epsilon}{d+\xi} \Big) \text{min} \{ \delta_{n_k}, 1-\delta_{n_k} \} \Big) (d+\xi)\\
		& = & \Big(1 - 2 \delta \Big(\frac{\epsilon}{d+\xi} \Big) \text{min} \{ \eta_{n_k} \delta_{n_k}, \eta_{n_k}(1-\delta_{n_k}) \} \Big) (d+\xi).
	\end{eqnarray*}
	Since we can find $l > 0$ such that $2 \,\text{min} \{ \eta_{n_k} \delta_{n_k}, \eta_{n_k}(1-\delta_{n_k}) \} \geq l,$
	\begin{eqnarray*}
		\Big(1 - 2 \delta \Big(\frac{\epsilon}{d+\xi} \Big)& \text{min} \{ \eta_{n_k} \delta_{n_k},\eta_{n_k}(1-\delta_{n_k}) \} \Big)(d+\xi)\\ & \leq \Big( 1- l\delta \Big(\frac{\epsilon}{d+\xi} \Big)\Big)(d +\xi).
	\end{eqnarray*}
	If we choose very small $\xi > 0$, we obtain $\Big( 1- l\delta \Big(\frac{\epsilon}{d+\xi} \Big)\Big)(d +\xi) < d$, a contradiction. This proves that $lim_{n \rightarrow \infty}\parallel w_n-Fw_n \parallel = 0.$  \\
	Since $F(M)$ is contained in a compact set, $\{Fw_n \}$ has a subsequence $\{ Fw_{n_k} \}$ that converges to a point $v_0 \in M.$ Also $\{ w_{n_k} \}$ converges to $v_0$. From the given sequence, we obtain
	\begin{eqnarray*}
		\parallel w_{n_k+1} - w_{n_k} \parallel &=& \parallel P^{n_k} \big((1-\eta_{n_k}) w_{n_k} + \eta_{n_k} F z_{n_k}\big) - w_{n_k} \parallel\\ 
		&\leq& \parallel (1-\eta_{n_k}) w_{n_k} + \eta_{n_k} F z_{n_k} - w_{n_k} \parallel \\
		&=& \eta_{n_k} \parallel w_{n_k} - F z_{n_k} \parallel.
	\end{eqnarray*}
	Since $\parallel F z_{n_k} - w_{n_k} \parallel \rightarrow 0,$ which implies 
	that $\parallel w_{n_k+1} - w_{n_k} \parallel \rightarrow 0.$ Therefore, $ w_{n_k+1} \rightarrow v_0,$ which implies that $w_n \rightarrow v_0.$ Also we have $F z_{n_k} \rightarrow v_0$ as $k \rightarrow \infty.$\\
	Now, $\parallel Fw_{n_k} - F(P_N(v_0)) \parallel \leq \parallel w_{n_k} - P_N(v_0) \parallel$ which gives that \\ 
	$\parallel v_0 - F(P_N(v_0)) \parallel \leq \parallel v_0 - P_N(v_0) \parallel.$  Therefore, $F(P_N(v_0)) = P_N(v_0)$.\\
	Also, $ \parallel F(P(v_0)) - P_N(v_0) \parallel = \parallel F(P(v_0)) - F(P_N(v_0)) \parallel \leq \parallel P(v_0) - P_N(v_0) \parallel.$ So $F(P(v_0)) = P(v_0)$.\\
	Now, $ \parallel FP_N(P(v_0)) - P(v_0) \parallel = \parallel FP_N(P(v_0)) - F(P(v_0)) \parallel \leq \parallel P_N(P(v_0)) - P(v_0) \parallel.$ Thus $FP_N(P(v_0)) = P_N(P(v_0))$.\\ 
	For any $n, F(P^n(v_0)) = P^n(v_0)$ and $FP_N(P^n(v_0)) = P_N(P^n(v_0)).$ By Theorem \ref{l6}, for each $w\in M$ the sequence $\{ P^n(w) \}$ converges to some $u(w) \in M_0.$ Now,
	\begin{eqnarray*}
		\parallel F(u(v_0)) -P_N(u(v_0)) \parallel & \leq & \lim_{n \rightarrow \infty} \parallel F(u(v_0)) -P_N(P^n(v_0)) \parallel\\
		& = & \lim_{n \rightarrow \infty} \parallel F(u(v_0)) -F(P_N(P^n(v_0))) \parallel\\
		& \leq & \lim_{n \rightarrow \infty} \parallel u(v_0) -P_N(P^n(v_0)) \parallel\\
		& = & \parallel u(v_0) -P_N(u(v_0)) \parallel.
	\end{eqnarray*}
	So $\parallel F(u(v_0)) -P_N(u(v_0)) \parallel \leq \parallel u(v_0) -P_N(u(v_0)) \parallel.$\\
	Therefore $F(u(v_0)) = u(v_0)$ and similarly $FP_N(u(v_0)) = P_N(u(v_0)).$\\
	Now we define $g_n : M \rightarrow \mathbb{R}$ by $g_n(w) = \parallel P^n(w) - u(w) \parallel.$\\ Since $\parallel u(w) - u(z) \parallel = \lim_{n \rightarrow \infty} \parallel P^n(w) - P^n (z) \parallel \leq \parallel w - z \parallel,$ then we conclude that $u$ is continuous. Therefore $g_n(w)$ is continuous and converges pointwise to zero. Since $u(w) \in M_0,$ by Lemma \ref{l11}, we obtain $g_{n+1} \leq g_n.$ Therefore $g_n$ converges uniformly on the compact set $$S = \{(1-\eta_{n_k})w_{n_k} + \eta_{n_k} Fz_{n_k}\} \cup \{v_0\}.$$ Therefore $$\lim_{k \rightarrow \infty} \parallel P^{n_k} ((1-\eta_{n_k})w_{n_k} + \eta_{n_k} Fz_{n_k}) - u((1-\eta_{n_k})w_{n_k} + \eta_{n_k} Fz_{n_k}) \parallel = 0.$$ Since $u((1-\eta_{n_k})w_{n_k} + \eta_{n_k} Fz_{n_k}) \rightarrow u(v_0),$ we get $w_{n_k+1} \rightarrow u(v_0),$ which gives that $u(v_0) = v_0.$ 
	Therefore $ Fv_0 = F(u(v_0)) = u(v_0) = v_0$ , which completes the proof.
\end{proof}

Suppose $X$ is a Hilbert space and let $F$ be as in Theorem \ref{l1}. 
Consider $P_M F : M \rightarrow M$ and $P_N F : N \rightarrow N$. 
From the Proposition \ref{l10}, $\parallel P_M F(w) - P_N F(z) \parallel \leq \parallel w-z \parallel$ for $w \in M$ and $z \in N,$ 
by Theorem \ref{l12} and Theorem \ref{l13} we give the following results on convergence of best proximity points.

\begin{corollary} \label{l14}
	Let $M$ and $N$ be nonempty, closed, bounded and convex subsets of a Hilbert space $X$. 
	Let $F$ be as in Theorem \ref{l1}. If $F(M)$ is mapped into a compact subset of $N$, 
	then for any $w_0 \in M_0$ the sequence defined by $w_{n+1} = (1-\eta_n) w_n + \eta_n P_M\big(F((1-\delta_n)w_n+\delta_n P_MFw_n)\big)$ converges 
	to $w$ in $M_0$ such that $\parallel w-Fw \parallel = d(M,N).$
\end{corollary}

\begin{corollary} 
	Let $M$ and $N$ be nonempty, closed, bounded and convex subsets of a Hilbert space $X$. 
	Let $F$ be as in Theorem \ref{l1}. If $F(M)$ is mapped into a compact subset of $N$, 
	then for any $w_0 \in M$ the sequence defined by $w_{n+1} = (1-\eta_n) w_n + \eta_n P_M\big(F((1-\delta_n)w_n+\delta_n P_MFw_n)\big)$ converges 
	to $w$ in $M_0$ such that $\parallel w-Fw \parallel = d(M,N),$ provided $d(w_n, M_0) \rightarrow 0.$
\end{corollary}

\begin{corollary}
	Let $M$ and $N$ be nonempty, closed, bounded and convex subsets of a Hilbert space $X$. 
	Let $F$ be as in Theorem \ref{l1}. If $F(M)$ is mapped into a compact subset of $N$, then for any $w_0 \in M_0$ the sequence defined by $w_{n+1} = P^n \big( (1-\eta_n) w_n + \eta_n P_M\big(F((1-\delta_n)w_n+\delta_n P_MFw_n)\big) \big)$ converges to $w$ in $M_0$ such that $\parallel w-Fw \parallel = d(M,N).$
\end{corollary}

\begin{proof}
	The result follows by Corollary \ref{l14}.
\end{proof}

\begin{corollary}
	Let $M$ and $N$ be nonempty, closed, bounded and convex subsets of a Hilbert space $X$. 
	Let $F$ be as in Theorem \ref{l1}. Let $w_0\in M,$ and define $w_{n+1}=P^n \big( (1-\eta_n)w_n + \eta_n P_MFz_n \big) ,\, $ where $  z_n = (1-\delta_n)w_n+\delta_n P_MFw_n,\, \eta_n, \delta_n \in (\epsilon, 1-\epsilon),$ 
	where $\epsilon\in (0,1/2)$ and $n=0,1,2,...$. if $F(M)$ is mapped into a compact subset of $N$ and $\parallel w_n - P_MF z_n \parallel \rightarrow 0$, then $\{w_n\}$ converges to $w$ in $M_0$ such 
	that $\parallel w-Fw \parallel = d(M,N).$
\end{corollary}

\begin{proof}
	The result follows by Theorem \ref{l13}.
\end{proof}



\label{lastpage}

\end{document}